\documentclass[a4paper,12pt,reqno]{amsart}
\usepackage{amsmath,amssymb,amsthm}
\usepackage{hyperref} 
\hypersetup{
    colorlinks=true,
    linkcolor=blue,   % internal links (like \cite)
    urlcolor=blue,    % URLs
    citecolor=blue    % citation color
}
\usepackage{ifthen}
\usepackage{cite}
\usepackage{graphicx}
\usepackage{float}
\usepackage{tcolorbox}
\usepackage{xcolor}
\usepackage{diagbox}
\usepackage{mathrsfs}
\usepackage{setspace}
\usepackage[a4paper, top=1in, bottom=1in, left=1in, right=1in]{geometry} 
\usepackage{ragged2e}
\theoremstyle{plain}

\newtheorem{thm}{Theorem}[section]
\newtheorem*{thm*}{Theorem}

\numberwithin{equation}{section}

\numberwithin{subcase}{case}

\newtheorem{lem}{Lemma}[section]
\newtheorem{prop}{Proposition}[section]

\newtheorem{defn}{Definition}[section]

\theoremstyle{definition}
\newcounter {own}
\def\theown {\thesection  .\arabic{own}}
\usepackage{tcolorbox}
\usepackage{xcolor}

{\qed\bigskip}

\newcounter{alphabet}

\newcounter{minutes}
\divide\time by 60
\newcounter{hours}
\multiply\time by 60 \addtocounter{minutes}{-\time}
\title{quaternion quadratic-phase Dunkl transform: Its properties and uncertainty principles}

\author{Akhilesh Prasad}
\author{Debabrata Biswal$^*$}
\author{Manab Kundu$^*$}
\address{Akhilesh Prasad, Indian Institute of Technology Dhanbad, Dhanbad-826 004, India.}
\email{aprasad@iitism.ac.in
}
\address{Debabrata Biswal, Indian Institute of Technology Dhanbad, Dhanbad-826 004, India.}
\email{debabrata.iitk@gmail.com}

\address{ Manab Kundu , SRM University A.P., Amaravati-52240, INDIA}
\email{manabiitism17@gmail.com} 

\date{}

\keywords{Quaternion Transform, The Dunkl Transform, Quadratic-Phase Dunkl Transform,  Heisenberg and Donoho–Stark Uncertainty Principle, Parseval's Formula. \\
$*$~Corresponding author}  

\begin{document}

\begin{abstract}
Motivated by the wide applications of the quadratic-phase integral transform in signal processing, optics, and time--frequency analysis, in this paper we develop a new integral transform called the quaternion quadratic-phase Dunkl transform (QQPDT). We present a detailed study of its fundamental properties. In particular, we establish continuity, linearity, scaling, modulation, Riemann-Lebesgue result, inversion formula, and Parseval’s identity.  In addition, we establish Heisenberg and Donoho-Stark-type uncertainty principles for the QQPDT. Finally, we develop a signal recovery algorithm in the QQPDT domain. The proposed method enables the reconstruction of missing signal information from partial observations under suitable localization conditions.
\end{abstract}
\maketitle
\pagestyle{myheadings}
\markboth{Akhilesh Prasad,  Debabrata Biswal, and Manab Kundu}{Two-Sided Quaternion Quadratic-Phase Dunkl Transform}

\section{Introduction}
In recent years, Dunkl harmonic analysis has attracted considerable attention due to its ability to incorporate reflection symmetries into classical harmonic analysis through Dunkl operators and the associated Dunkl kernel \cite{Dunkl1991, Rosler2003}. The mathematical structure and broad applicability of the Dunkl transform have motivated researchers to develop several generalizations and variants, such as fractional, linear canonical and other Dunkl type transforms \cite{umamaheswari2026, Tyr2025, Ghazouani2016, deJeu2006}. on the other hand, the quadratic-phase Dunkl transform (QPDT) has emerged as an important extension of the classical Dunkl transform by incorporating quadratic-phase factors into its kernel and reflection symmetries. This extension provides greater flexibility and effective for the analysis of non stationary and chirp type signals that frequently arise in radar, sonar, and optical applications. It has attracted increasing interest in harmonic analysis and successfully applied to signal processing, optics, image analysis, and mathematical physics \cite{IMAGEPROCESSING, colourimage, Raj Kumar, PrasadKundu2023}.

In modern applications, Quaternion analysis has emerged as a powerful tool for the 
representation and processing of multidimensional signals 
\cite{Quaternion book, IMAGEPROCESSING}. It decomposes signals into
four components, there by providing a richer representation of symmetry properties. Several quaternion valued integral 
transforms including  Fourier transform (QFT) \cite{mypaper}, windowed Fourier transforms (QWFT) \cite{qsignal, PrasadKundu2023}, fractional Fourier transforms (QFrFT) \cite{fractionalqdt}, quadratic-phase Fourier transform (QQPFT) \cite{Aprasad}, linear canonical transforms (QLCT) \cite{Kou2016, Bahari2019, PrasadKundu2023}, and Dunkl Transform (QDT) \cite{Bhat2026, TWOSIDEDQDT, Tyr2025} and several others have been developed together with uncertainty principles \cite{Dar2023, Safouane2026, Bahri2016, chen2015}. Due to the non commutative nature of 
quaternion multiplication, three distinct versions of quaternion 
transform can be defined such as the left-sided, right-sided, and two-sided which provide a 
powerful 
framework for modeling three dimensional rotations and 
orientations. In this paper, we 
focus on the two-sided quaternion transform. It consequently
played a pivotal role in advances across robotics, signal processing, computer graphics,
color image processing, aerospace engineering \cite{colourimage, IMAGEPROCESSING, Quaternion book}.    

Although substantial progress
has been achieved independently in quadratic-phase analysis, Dunkl
analysis, and quaternion
harmonic analysis, the combination of these three
frameworks remains largely 
unexplored. Motivated by this gap, in this pape we 
introduce a new integral transform called QQPDT. The proposed transform combines quaternionic structure with the reflection symmetry of the Dunkl framework and the additional flexibility provided by quadratic-phase factors. This enriched kernel enables the QQPDT to provide a unified and versatile framework for the analysis of multidimensional, non-stationary, and chirp-type signals with applications in color image processing, image filtering, watermarking, edge detection, and pattern recognition, thereby extending the scope of existing quadratic-phase and quaternionic Dunkl transforms.
In addition, we establish the fundamental properties of the proposed transform. Parallel to these structural development, uncertainty principles are among the most fundamental results in harmonic analysis \cite{UPQDT, Ghazouani2016, Sahbani, Shah2021, PrasadKundu2023, Shah2021}. Further, we establish Heisenberg and
Donoho–Stark type uncertainty principles with signal 
recovery algorithm for the QQPDT \cite{chen2015, Sahbani, Dar2023}. 

The key contributions of the present paper are as follows:
\begin{itemize}
\item We introduce the QQPDT as a new extension of the quadratic-phase Dunkl transform in the quaternion setting and establish its fundamental analytical properties.
\item We establish the Heisenberg and Donoho--Stark uncertainty principles for the QQPDT and develop a signal recovery algorithm in the QQPDT domain.
\end{itemize}
The paper is organized as follows. In Section~2, we recall
the essential preliminaries on quaternion algebra and Dunkl 
analysis. In Section~3, we introduce the two-sided QQPDT, along with its important special
cases corresponding to real parameters and establish its 
fundamental properties which include continuity, Riemann Lebesgue lemma, inversion formula, and Parseval's identity.
Section~4 is devoted to the formulation of a Heisenberg-type uncertainty principle in the quaternion quadratic-phase Dunkl setting. Section~5, develops Donoho–Stark uncertainty principles for the QFDT and QQPDT. As an important 
application of the proposed framework, Section~6 develops a
signal recovery algorithm in the QQPDT domain.

\section{Preliminaries} 
This section provides basic results of Dunkl Harmonic Analysis and a brief overview of  quaternion and  two sided quaternion Dunkl harmonic analysis. 
{\subsection{Dunkl Harmonic Analysis} 
Let $R$ be a root system in $\mathbb{R}$ and $\mathbb{Z}_2$ the associated reflection
group. The multiplicity function $\mu:R\rightarrow\mathbb{C}$ satisfies
$\mu(\alpha)=\mu(w(\alpha))$ for all $\lambda\in\mathbb{Z}_2$. Throughout we assume
$\mu \geq -\frac{1}{2}.$
The Dunkl weight function on $\mathbb{R}$ is defined as}
{\begin{equation*}
\lambda_\mu(x)=|x|^{2\mu+1}
\end{equation*}}
which is invariant under the action of $\mathbb{Z}_2$ and homogeneous of degree
$2\mu+1$, and the Mehta-type constant $c_\mu$ defined by 
{\begin{equation}\label{2.111}
c_\mu=\left(\int_{\mathbb{R}}
\exp\left(-\frac{|x|^2}{2}\right) \lambda_\mu(x)\,dx
\right)^{-1}.
\end{equation}

The associated measure on $\mathbb{R}^2$ is}
{\begin{equation} \label{eq2.11}
d\lambda_\mu(x)=w_\mu(x_1)w_\mu(x_2)\,dx_1dx_2.
\end{equation}
For essential results in Dunkl harmonic analysis on the real line, we refer the reader to \cite{Rosler2003, Dunkl1991}.

\subsection{Quaternion}
It is a class of hyper-complex numbers that was formally introduced 
by William Rowan Hamilton in 1843.
It extends the concept of complex field $\mathbb{C}$ to a four-dimensional algebra, denoted by 
$\mathbb{H}$. Every quaternion $q\in \mathbb{H}$ can be expressed in the 
form
$ \textbf{q}=q_0+ \textbf{\textbf{i}}q_1+ \textbf{\textbf{j}}q_2+ \textbf{k}q_3 $ where $q_0,q_1,q_2,q_3 \in \mathbb{R}$.
The three imaginary units 
$\bf{\textbf{i}},\bf{\textbf{j}}, \bf{k}  $ satisfy specific non-commutative multiplication rules, which characterize the algebraic structure of the quaternion system.
\begin{equation*} \label{eq1.1}
\bf{ij} = -\bf{ji} = -\bf{k}, \quad \bf{jk} = -\bf{kj} = \bf{\textbf{i}}, \quad \bf{ki} = -\bf{\textbf{i}k} = \bf{\textbf{j}}, \quad \bf{\textbf{i}}^2 = \bf{\textbf{j}}^2 = \bf{k}^2 = \bf{ijk} = -1, 
\end{equation*}
For a quaternion $ \textbf{q} = q_0 + \textbf{\textbf{i}}q_1 + \textbf{\textbf{j}}q_2 + \textbf{k}q_3 \in \mathbb{H}$, $q_0$ is called the scalar part of $q$ denoted by $Sc(q)$ and a pure quaternion $\bf{q}$ denoted by $\text{Vec}(\textbf{q}) = \textbf{\textbf{i}}q_1 + \textbf{\textbf{j}}q_2 + \textbf{k}q_3$, and for $|\textbf{q}|=1$, \textbf{q} is called unit quaternion.
The multiplication of two quaternions is written 
 as
$\textbf{qp} = q_0 p_0 - \textbf{q} \cdot \textbf{p}+ q_0 \textbf{p} + p_0 \textbf{q} + \textbf{q} \times \textbf{p},$ 
where
\begin{eqnarray}
 \textbf{q} \cdot \textbf{p} &=& q_1 p_1 + q_2 p_2 + q_3 p_3, \notag \\
\textbf{q} \times \textbf{p} &=& \mathbf{\textbf{i}}(q_2 p_3 - q_3 p_2) + \mathbf{\textbf{j}}(q_3 p_1 - q_1 p_3) + \mathbf{k}(q_1 p_2 - q_2 p_1). \notag  
\end{eqnarray}
For quaternions, multiplication is non-commutative \textbf{i}.e $\textbf{pq} \neq \textbf{qp} $ \\
 The quaternionic conjugation $\bar{\textbf{q}}$ is given by 
\begin{equation*}\label{eq1.3}
\bar{\textbf{q}} = q_0 - \mathbf{\textbf{i}}q_1 - \mathbf{\textbf{j}}q_2 - \mathbf{k}q_3.
\end{equation*}
Using the above properties, we have
$\overline{\textbf{qp}} = \overline{\textbf{p}} \, \overline{\textbf{q}}.$

Any quaternion can be represented by two complex number. 
\begin{eqnarray*}
\textbf{q}&=&(q_0+ \mathbf{\textbf{i}}q_1)+ (\mathbf{\textbf{j}}q_2+ \mathbf{k}q_3) \\
&=&(q_0+ \mathbf{\textbf{i}}q_1)+ (\mathbf{\textbf{j}}q_2- \mathbf{ji}q_3) \\
&=&(q_0+ \mathbf{\textbf{i}}q_1)+ \mathbf{\textbf{j}}(q_2- \mathbf{\textbf{i}}q_3)\\
&=& z_1 + \mathbf{\textbf{j}}\overline{z_2}
\end{eqnarray*}
where $z_1, z_2 \in \mathbb{C}$ are two complex numbers.\\
\noindent and its norm is defined as
\begin{equation*}
|\textbf{q}| = \sqrt{q\overline{q}} = \sqrt{q_0^2 + q_1^2 + q_2^2 + q_3^2}.
\end{equation*}
Also it is easy to check that
$|pq| = |p||q|, \quad p, q \in \mathbb{H}.$ 
In addition,the real scalar component exhibits cyclic symmetry under multiplication
\begin{equation*}
\langle pqr \rangle = \langle qrp \rangle = \langle rpq \rangle, \quad \forall p, q, r \in \mathbb{H}.
\end{equation*}
From the quaternion conjugate and the modulus of q, the inverse of a non-zero quaternion is given by $p\in \mathbb{H}$ is $ p^{-1} = \frac{\overline{p}}{|p|^2} $
    
\begin{defn}
 A mapping $f : \mathbb{R}^2 \to \mathbb{H}$  with quaternion values can be represented as
\begin{equation*}
f(t_1, t_2) = f_1(t_1, t_2) + \mathbf{\textbf{i}}f_2(t_1, t_2) + \mathbf{\textbf{j}}f_3(t_1, t_2) + \mathbf{ij}f_4(t_1, t_2),   
\end{equation*}
where  $f_r : \mathbb{R}^2 \to \mathbb{R}$ \textit{for} $r = 1, 2, 3, 4$. 
\end{defn}

\begin{defn}{Quaternion Schwartz space $\mathcal{S}(\mathbb{R}^2,\mathbb{H})$: \cite{Safouane2026}}  
The space consists of all quaternion valued
functions with finite $\|\cdot\|_{p,\mu}$ norm.
The quaternion Schwartz space $\mathcal{S}(\mathbb{R}^2,\mathbb{H})$ is defined as the set of
$C^\infty$ functions from $\mathbb{R}^2$ to $\mathbb{H}$ such that for all multi-indices
$\alpha$ and $m\in\mathbb{N}$,
\begin{equation*}
\sup_{x\in\mathbb{R}^2}
(1+\|x\|)^m |D^\alpha f(x)| < \infty.
\end{equation*}
\end{defn}

\begin{defn} 
\textit{For $p \in [1, \infty)$, $L^p_\mu(\mathbb{H})$ is the space of Lebesgue measurable function $f$ on $ \mathbb{H}$ is defined as \cite{Safouane2026, Saoudi2023} }
\begin{equation*}
\|f\|_{\mu,p} = 
\begin{cases} 
\left( \int_{\mathbb{R}^2} |f(v)|^p |v|^{2\mu+1} dv \right)^{\frac{1}{p}}< \infty & \text{for} \quad 1 \le p < \infty, \\
\text{ess sup}_{t \in \mathbb{R}} |f(v)| < \infty & \text{for} \quad p = \infty.
\end{cases}
\end{equation*}
\end{defn}

\begin{defn}
The inner product of two quaternion-valued functions $f$ and $g$ is defined as \cite{Bhat2026}
\begin{eqnarray*}
\langle f,g\rangle_{\mathbb{H}}
=
\int_{\mathbb{R}^{2}}
f(v)\,\overline{g(v)}\,d\lambda_\mu(v),
\end{eqnarray*}
where the weight measure $d\lambda_\mu(v)$ is given by \eqref{eq2.11}
\begin{eqnarray*}
d\lambda_\mu(v)
= \lambda_{\mu}(v_1)\lambda_{\mu}(v_2) dv_1 dv_2= |v_1|^{2\mu+1} |v_2|^{2\mu+1} \, dv_1 \, dv_2,
\quad \textit{where} \quad
v=(v_{1},v_{2})\in\mathbb{R}^{2}.
\end{eqnarray*}
\end{defn}
In particular, for $f=g$, we obtain the norm associated with the inner product
$\langle\cdot,\cdot\rangle_{\mathbb{H}}$ given by
\begin{eqnarray*}
\|f\|_{2,\mu}
=
\left(
\int_{\mathbb{R}^{2}}
|f(v)|^{2}
\,d\lambda_\mu(v)
\right)^{\frac12}.
\end{eqnarray*}

If $f \in L^2(\mathbb{R}^2, \mathbb{H})$ then define as
\begin{equation*}
L^2(\mathbb{R}^2, \mathbb{H})=\left\{ f| f :\mathbb{R}^2 \rightarrow \mathbb{H} ,\|f\|< \infty\right\}.
\end{equation*}
The space $C_0(\mathbb{R}^2, \mathbb{H})$ consists of all continuous quaternion-valued functions defined on $\mathbb{R}^2$ that vanish at infinity. It is defined as:
\begin{equation*}
C_0(\mathbb{R}^2, \mathbb{H}) = \left\{ f \in C(\mathbb{R}^2, \mathbb{H}) : \lim_{|v| \to \infty} |f(v)| = 0 \right\}.
\end{equation*}
\subsection{The Two-Sided Quaternionic Dunkl Transform}

Let $x, y \in \mathbb{R}$. For \\ $\mu \ge -\frac{1}{2}$, the normalized Bessel function of the first kind $j_\mu$ is defined by
\[
j_\mu(x) = \Gamma(\mu+1)\sum_{n=0}^{\infty} \frac{(-1)^n}{n!\,\Gamma(n+\mu+1)}
\left(\frac{x}{2}\right)^{2n}.
\]
Using $j_\mu$, the generalized quaternionic exponential function defined as \cite{Bhat2026}
\begin{eqnarray}
E_{\mu}(\textbf{q}x, y) = j_\mu(x y) + \frac{\textbf{q} x y}{2(\mu+1)} j_{\mu+1}(x y),
\end{eqnarray}
where $\textbf{q} \in \{ \textbf{i}, \textbf{j}\}$ denotes a quaternionic imaginary unit, and satisfies \cite{Rosler2003}
\begin{equation} \label{eq2.100}
E_\mu(\textbf{\textbf{i}}w,v)=E_\mu(v,\textbf{\textbf{i}}w),\,  E_{-1/2}(\textbf{\textbf{i}}w, v)=e^{\textbf{\textbf{i}}wv}\, \textit{and} \quad \,|E_\mu(\textbf{\textbf{i}}w,v)| \leq 1 .
\end{equation}
\begin{defn}
Two-Sided QDT 
For a function 
$f \in L^1_{\mu}(R^2,\mathbb{H})$  is followed by \cite{Tyr2025}
\begin{equation}\label{eq1.12}
\mathcal{ D}^{\mathbb{H}}_{\mu}[f](w_1, w_2) = c_{\mu}^2 \int_{\mathbb{R}^2} E_{\mu}(-\textbf{i} w_1, v_1) \, f(v_1, v_2) \, E_{\mu}(-\textbf{j} w_2, v_2) \, d\lambda_\mu(v), 
\end{equation}
where $E_{\mu}(-\textbf{\textbf{i}}w_1, v_1)$ and $E_{\mu}(-\textbf{\textbf{j}} w_2, v_2)$ are the associated Dunkl kernels.
\end{defn}
When $\mu=-1/2$, the Dunkl kernel reduces to the complex exponential, and the QDT
reduce the classical two-sided QFT \cite{mypaper}.
\begin{thm}
If $f \in L^1(\mathbb{R}^2,\mathbb{H})$ and $\mathcal{D}^{\mathbb{H}}_\mu f \in L^1(\mathbb{R}^2,\mathbb{H})$, 
then the inverse of two sided QDT is followed by \cite{Tyr2025}
\begin{equation}
f(v_1, v_2) = c_{\mu}^2 \int_{\mathbb{R}^2} E_{\mu}(\textbf{i}\, w_1, v_1) \, \mathcal{D}^{\mathbb{H}}_{\mu}[f](w_1, w_2) \, E_{\mu}(\textbf{j} w_2, v_2) \, d\lambda_\mu(w). \label{eq2.10}
\end{equation}
For $\mu=-1/2$, it becomes the inverse of classical two-sided QFT \cite{mypaper}.
\end{thm}
\begin{thm} \label{thm2.2}
(Parseval's Formula) For all $f, g \in L^2_{\mu}(\mathbb{R}^2; \mathbb{H})$, the Parseval's formula for the two-sided QDT is defined as \cite{Bhat2026}
\begin{equation}\label{eq1.14}
\int_{\mathbb{R}^2} f(v_1, v_2) \overline{g(v_1, v_2)} \, d\lambda_\mu(v) 
= \int_{\mathbb{R}^2} \mathcal{D}^{\mathbb{H}}[f](w_1, w_2) \, \overline{\mathcal{D}^{\mathbb{H}}[g](w_1, w_2)} \, d\lambda_\mu(w),
\end{equation}
\end{thm}
\begin{thm}
(Plancherel's Formula) For all $f \in L^2_{\mu}(\mathbb{R}^2; \mathbb{H})$, the two-sided QDT satisfies \cite{Bhat2026}
\begin{equation}
\int_{\mathbb{R}^2} |f(v_1, v_2)|^2 \, d\lambda_\mu(v) = \int_{\mathbb{R}^2} \left| \mathcal{D}^{H}_{\mu}[f](w_1, w_2) \right|^2 \, d\lambda_\mu(w).
\end{equation}
\end{thm}
\section{Two-Sided QQPDT and Its Fundamental Properties}
This section is organized into two subsections. In the first subsection, we introduce the two-sided QQPDT and discuss its key special cases corresponding to real parameters $(a_1, b_1, c_1, d_1, e_1, a_2, b_2, c_2, d_2, e_2)$ for two dimensions with the condition $b_1, b_2 \neq 0$, together with the Dunkl indices $\mu \ge -\frac{1}{2}$. In the second subsection, we establish the fundamental properties of the QQPDT, including the inversion formula and Parseval’s identity, which confirm its internal consistency and energy-preserving nature. Throughout the paper, we use the notation $u = (u_1, u_2)$, $w = (w_1, w_2), \quad \textit{and} \quad v = (v_1, v_2) \in \mathbb{R}^2$.
\subsection{Definition and Special Cases of the Two-Sided QQPDT}
\begin{defn} 
For $f \in L^1_\mu(\mathbb{R}^2,\mathbb{H})$ the QQPDT is defined as 
\begin{eqnarray}\label{eq2.1}
\mathbb \mathbb{D}^{a,b,c}_{d,e,\mu}[f](\lambda) = c_{\mu}^2 \int_{\mathbb{R}^2} \Psi^{a_1, b_1, c_1}_{d_1, e_1, \mu}(w_1, v_1) ~f(v_1, v_2)~ \Psi^{a_2, b_2, c_2}_{d_2, e_2, \mu}(w_2, v_2)~d\lambda_\mu(v),   
\end{eqnarray}
where\quad $\Psi^{a_1, b_1, c_1}_{d_1, e_1, \mu}$ and $\Psi^{a_2, b_2, c_2}_{d_2, e_2, \mu}$ are denote the QQPD kernel, given by
\begin{eqnarray}\label{eq2.2}
\Psi^{a_1, b_1, c_1}_{d_1, e_1, \mu} (w_1, v_1) & = & \frac{ E_{\mu}(-\textbf {i}w_1 /{b_1}, v_1)}{(\textbf{i}\,b_1)^{\mu+1} }  e^{-\textbf{i}(a_1 v_1^2 + c_1  w_1^2 + d_1 v_1 + e_1  w_1)}  \\ 
\textit{and} \quad \Psi^{a_2, b_2, c_2}_{d_2, e_2, \mu}(w_2, v_2) & = &  \frac{ E_{\mu}(-\textbf {j}w_2/{b_2}, v_2)}{ (\textbf {\textbf{j}}\,b_2)^{\mu+1}}  e^{-\textbf{j}(a_2 v_2^2 + c_2  w_2^2 + d_2 v_2 + e_2  w_2)}, \label{eq2.32}
\end{eqnarray}
with $E_{\mu}(-\textbf{i}w_1/b_1, v_1)$ and $E_{\mu}(-\textbf{j}w_2/b_2, v_2)$ are the Dunkl kernels and $\textbf{i}, \textbf{j}$ are the unit quaternions.
\end{defn}
\textbf{Particular cases of   QQPDT:} 
\textbf{Case (\textbf{i}).} if ${\mu}=-\frac{1}{2}$, then under the transform $b_1\rightarrow \frac{1}{b_1}$ and $b_2\rightarrow\frac{1}{b_2}$ QQPDT reduce to QQPFT \cite{ mypaper}.\\
For ${\mu}=-\frac{1}{2}$ we have $c_{\mu}= \frac{1}{\sqrt{2\pi}}$, 
$E_{-1/2}(-\textbf{i}w_1 , v_1) = e^{-\textbf{i}w_1 v_1}$ and $ E_{-1/2}(-\textbf{j}w_2 b_2, v_2) = e^{-\textbf{j}w_2 v_2}$
then \eqref{eq2.1} becomes
\begin{eqnarray*}
\mathcal{Q}_{d,e,{\mu}}^{a,b,c} [f](w_1, w_2)
&=& \sqrt{\frac{b_1 \textbf{i}}{2\pi}}  \int_{\mathbb{R}^2} e^{-\textbf{i}(a_1 v_1^2 + c_1 w_1^2 + b_1 w_1 v_1 + d_1 v_1 + e_1 w_1)} \\
&\times& f(v_1, v_2)\,\sqrt{\frac{b_2 \textbf{j}}{2\pi}} e^{-\textbf{j}(a_2 v_2^2 + c_2 w_2^2 + b_2 w_2 v_2 + d_2 v_2 + e_2 w_2)}  \, dv_1 \, dv_2.
\end{eqnarray*}
\textbf{Case (II).} If $a_1 = c_1 = d_1 = e_1 = 0 ,b_1=1$ 
and $a_2 = c_2 = d_2 = e_2 = 0 ,b_2=1,$ then 
the QQPDT is reduced to quaternion Dunkl transform \cite{TWOSIDEDQDT}.

\textbf{Case (III).} If ${\mu} = -1/2$, $a_1 = c_1 =d_1 = e_1 = 0$, $b_1=1$, and $a_2 = c_2 =d_2 = e_2 = 0$, $b_2=1$, then QQPDT reduces to the \emph{quaternion  Fourier transform} \cite{mypaper}.

\textbf{Case (IV).} If ${\mu} = -1/2$, $d_1 = e_1 = d_2 = e_2 = 0$, $b_1,b_2 \neq 0$, under the transformation $a_1 \to -\frac{a_1}{2b_1}$ and $c_1 \to -\frac{c_1}{2b_1}$, $a_2 \to -\frac{a_2}{2b_2}$ and $c_2 \to -\frac{c_2}{2b_2}$, then the QQPDT reduces to the QLCT \cite{Bahari2019}.
\begin{eqnarray*}
\mathcal{Q}[f](w_1, w_2) &=& \frac{1}{\sqrt{2\pi \textbf{i} b_1}}  \int_{\mathbb{R}^2} e^{\textbf{i}\left(\frac{a_1 v_1^2 + c_1 w_1^2}{2b_1}\right)} E_{{\mu}}\left(-	\textbf{i}w_1/b_1, v_1\right) f(v_1, v_2) \\
&\times&  \frac{1}{\sqrt{2\pi \textbf{j} b_2}} E_{{\mu}}\left(-\textbf{j}w_2/b_2, v_2\right) e^{\textbf{j}\left(\frac{a_2 v_2^2 + c_2 w_2^2}{2b_2}\right)} \, dv_1 \, dv_2  \\
&=& \frac{1}{\sqrt{2\pi \textbf{i} b_1}}  \int_{\mathbb{R}^2} e^{\frac{\textbf{i}}{2b_1}\left({a_1 v_1^2 + c_1 w_1^2 - 2w_1 v_1}\right)}  f(v_1, v_2)\\&\times& \frac{1}{\sqrt{2\pi \textbf{j} b_2}} e^{\frac{\textbf{j}}{2b_2}\left({a_2 v_2^2 + c_2 w_2^2 - 2w_2 v_2}\right)} \, dv_1 \, dv_2.
\end{eqnarray*}

\textbf{Case(V).} If $ d_1 = d_2 = e_1 = e_2 = 0 $ and $a_1=c_1=-\frac{\cot\theta_1}{2}$, $a_2=c_2=-\frac{\cot\theta_2}{2}$
$b=\sin\theta$ where$(\theta \neq n\pi)$
then the QQPDT reduced to the QFrDT \cite{fractionalqdt}:
\begin{eqnarray*}
\mathcal{D}^{\theta}_{{\mu}}[f](w_1,w_2)
&=&A_{{\mu},\alpha_1}
\int_{\mathbb{R}^2}
e^{\frac{\textbf{i}}{2}(v_1^2+w_1^2)\cot\theta_1}
E_{{\mu}}\!\left(\frac{-	\textbf{i}w_1}{\sin\theta_1},v_1\right)
f(v_1,v_2) A_{{\mu},\theta_2} \\
&\times&
E_{{\mu}}\!\left(\frac{-\textbf{j}w_2}{\sin\theta_2},v_2\right)
e^{\frac{\textbf{j}}{2}(v_2^2+w_2^2)\cot\theta_2}
\,~d\lambda_\mu(v)
\end{eqnarray*}
\begin{eqnarray*}
\quad \textit{where} \quad  A_{{\mu},\theta_1} =\frac{e^{\textbf{i}({\mu}+1)(\hat{\theta}_1\pi/2-(\theta_1-2n\pi))}} {\Gamma({\mu}+1)\,(2|\sin\theta_1|)^{{\mu}+1}}, \quad where \quad \, \hat {\theta_1} =sgn(\sin\theta_1) \\ \textit{and} \quad
A_{{\mu},\theta_2} = \frac{e^{\textbf{j}({\mu}+1)(\hat{\theta}_2\pi/2-(\theta_2-2n\pi))}} {\Gamma({\mu}+1)\,(2|\sin\theta_2|)^{{\mu}+1}}, \quad where\quad \hat{\theta_2} =sgn(\sin\theta_2).
\end{eqnarray*}
\textbf{Case (VI).} If ${\mu} = -1/2$, $d_1 = e_1 = 0$, $d_2 = e_2 = 0$, and
\begin{eqnarray*}
a_1 &= c_1 = \frac{-\cot \theta_1}{2}, \quad b_1 = \csc \theta_1 \quad (\theta_1 \neq n\pi, n \in \mathbb{Z}), \\
a_2 &= c_2 = \frac{-\cot \theta_2}{2}, \quad b_2 = \csc \theta_2 \quad (\theta_2 \neq n\pi, n \in \mathbb{Z}),
\end{eqnarray*}
then the QQPDT (amplyfing by $\sqrt{1 - \textbf{i} \cot \theta_1}$ and $\sqrt{1 - \textbf{j} \cot \theta_2}$) reduce to the quaternion fractional Fourier transform (QFrFT) \cite{fractionalqdt}:
\begin{equation*}
\mathcal{F}^Q [f(v_1, v_2)](w_1, w_2) = \int_{\mathbb{R}^2} K_{\theta_1}(w_1, v_1) f(v_1, v_2) K_{\theta_2}(w_2, v_2) \, dv_1 \, dv_2,
\end{equation*}
where $K_{\theta_1}(w_1, v_1)$ and $K_{\theta_2}(w_2, v_2)$ are the kernels of QFrFT given by:
\begin{eqnarray*}
K_{\theta_1}(w_1, v_1) &= \sqrt{\frac{1 - \textbf{i} \cot \theta_1}{2\pi}} \exp \left\{ \frac{\textbf{i}}{2} (w_1^2 + v_1^2) \cot \theta_1 - \textbf{i} w_1 v_1 \csc \theta_1 \right\} \\ \textit{and} \quad
K_{\theta_2}(w_2, v_2) &= \sqrt{\frac{1 - \textbf{j} \cot \theta_2}{2\pi}} \exp \left\{ \frac{\textbf{j}}{2} (w_2^2 + v_2^2) \cot \theta_2 - \textbf{j} w_2 v_2 \csc \theta_2 \right\}
\end{eqnarray*}
The following case has not yet been explored. \\
\textbf{Case (VII).} If $d_1 = e_1 = 0$, $b_1,b_2 \neq 0$ and under the transformation:
\begin{equation*}
a_1 \to \frac{-a_1}{2b_1},\quad c_1 \to \frac{-c_1}{2b_1},
\quad \textit{and} \quad
a_2 \to \frac{-a_2}{2b_2},\quad c_2 \to \frac{-c_2}{2b_2}
\end{equation*}
then the QQPDT reduces to \emph{quaternion linear canonical Fourier-Bessel transform} defined as:
\begin{eqnarray*}
\mathcal{L}^B_{{\mu}} [f(v_1, v_2)](w_1, w_2) &=& c_{\mu}^2  \int_{\mathbb{R}^2} K_{{\mu}}(w_1, v_1) f(v_1, v_2) K_{{\mu}}(w_2, v_2) d\lambda_\mu(v), 
\end{eqnarray*}
 where $K_{\mu}$ and $K_{\mu}$ are the canonical Fourier-Bessel kernel, defined by 
\begin{eqnarray*}
K_{{\mu}}(w_1,v_1)
&=& \frac{1}{(\textbf{i}b_1)^{\mu+1}}
e^{\frac{\textbf{i}}{2}\left(\frac{a_1}{b_1}v_1^2+\frac{d_1}{b_1}w_1^2\right)}
J_{{\mu}}\!\left(\frac{w_1v_1}{b_1}\right),
\\ \textit{and} \quad
K_{{\mu}}(w_2,v_2)
&=& \frac{1}{(\textbf{j}b_2)^{\mu+1}} 
e^{\frac{\textbf{j}}{2}\left(\frac{a_2}{b_2}v_2^2+\frac{d_2}{b_2}w_2^2\right)}
J_{{\mu}}\!\left(\frac{w_2v_2}{b_2}\right).
\end{eqnarray*}
where $J_{\mu}$ is the sperical bessel function.

\subsection{Fundamental Properties of QQPDT}
\begin{lem}
Let $a_1,a_2, b_1, b_2, c_1, c_2, d_1, d_2, e_1, e_2 \in \mathbb{R}$, such that\\ $b_1 \neq 0$, $b_2 \neq 0$ with $\mu \geq -\tfrac{1}{2}$ then ,we have
\begin{enumerate}
\item The kernels of QQPDT are bounded 
\begin{eqnarray*}
\left|\Psi^{a_1,b_1,c_1}_{d_1,e_1,\mu}(w_1,v_1) \right| \leq \frac{1}{|b_1|^{\mu+1}} &\text{and}& \left| \Psi^{a_2,b_2,c_2}_{d_2,e_2,\mu}(w_2,v_2) \right| \leq \frac{1}{|b_2|^{\mu+1}} 
\end{eqnarray*}
\item{Conjugation property}
\begin{eqnarray*}
\overline{\Psi^{a_1, b_1, c_1}_{d_1, e_1, \mu} (w_1, v_1)}
&=& \Psi^{-a_1, -b_1, -c_1}_{-d_1, -e_1, \mu} (w_1, v_1) = \Psi^{-c_1,-b_1,-a_1}_{-e_1,-d_1,\mu}(v_1,w_1) \\
\text{and}\qquad \overline{\Psi^{a_2, b_2, c_2}_{d_2, e_2, \mu}(w_2, v_2)} &=& \Psi^{-a_2, -b_2, -c_2}_{-d_2, -e_2, \mu}(w_2, v_2) =\Psi^{-c_2,-b_2,-a_2}_{-e_2,-d_2,\mu}(v_2,w_2).\\
\end{eqnarray*}
\end{enumerate}
\end{lem}
\begin{proof}
The results follow immediately by applying \eqref{eq2.100}, \eqref{eq2.2} and \eqref{eq2.32}.
\end{proof}

\begin{prop} 
Let
$\alpha,\beta \in \mathbb{R}$
then for every $f,g \in L^1_\mu(\mathbb{R}^2,\mathbb{H})$, the two sided QQPDT satisfies the following properties:
\begin{enumerate}
\item \textbf{Linearity}
\begin{eqnarray*}
\mathbb{D}^{a,b,c}_{d,e,\mu}[\alpha f + \beta g](w)
=\alpha\,\mathbb{D}^{a,b,c}_{d,e,\mu}[f](w)+\beta\, \mathbb{D}^{a,b,c}_{d,e,\mu}[g](w).
\end{eqnarray*}
\item \textbf{Scaling}
Assume the scaling parameter $k > 0$ , and define
\[
f_k(v_1,v_2)=f(k v_1,k v_2).
\]
 Then
\begin{eqnarray*}  
\mathbb{D}^{a,b,c}_{d,e,\mu}[f_k](w_1,w_2)
&=&k^{-2\mu-2}
\mathbb{D}^{a',b',c}_{d',e,\mu}[f](w_1,w_2),
\end{eqnarray*}
where
\begin{eqnarray*}
a' =
\left(\frac{a_1}{k^2},\frac{a_2}{k^2}\right),
\quad
b' = (k b_1,k b_2),
\quad \textit{and} \quad
d' =\left(\frac{d_1}{k},\frac{d_2}{k}\right).
\end{eqnarray*}

\item \textbf{Modulation}
Let \(f\in L^1_\mu(\mathbb{R}^2,\mathbb{H})\) and let
\(\xi_1,\xi_2\in\mathbb{R}\). Then the quaternion modulation defined as $\mathcal{M}_{\xi_1,\xi_2}f(v_1,v_2)=e^{\mathbf{i}\xi_1v_1}
f(v_1,v_2)e^{\mathbf{j}\xi_2v_2}.$ Then
\begin{eqnarray*}
\mathbb{D}^{a,b,c}_{d,e,\mu}
\left[\mathcal{M}_{\xi_1,\xi_2}f\right](w_1,w_2)
=\mathbb{D}^{a,b,c}_{d',e,\mu}[f](w_1,w_2),
\end{eqnarray*}
where
\[
d'=(d_1-\xi_1,d_2-\xi_2).
\]
\end{enumerate}
\end{prop}

\begin{thm} \label{th3.1}
\begin{enumerate}
\item If $f \in L^1_{\mu}(\mathbb{R}^2,\mathbb{H})$, then the mapping
\begin{eqnarray*}
w \longmapsto \mathbb{D}^{a,b,c}_{d,e,\mu}[f](w)
\end{eqnarray*}
is continuous on $\mathbb{R}^2$ and we have
\begin{equation}
\left| \mathbb{D}^{a,b,c}_{d,e,\mu}[f](w) \right|
\leq\frac{c_{\mu}^2 }{|b_1 b_2|^{\mu+1}} \, \|f\|_{1,\mu}.
\end{equation}
\item The two sided QQPDT can be characterized through the quaternion Dunkl transform. In particular, for any function $f \in \mathcal{S}(\mathbb{R}^2,\mathbb{H})$,
 it can be written as
\begin{equation}\label{eq3.2}
\mathbb{D}^{a,b,c}_{d,e,\mu}[f](w)
= \frac{ e^{-\textbf{i}\left(c_1w_1^2+e_1w_1\right)} }{(\textbf{i}b_1)^{\mu+1}}
\mathcal{D}^{\mathbb{H}} _{\mu}(g)\!\left(\frac{w_1}{b_1},\frac{w_2}{b_2}\right) \frac{e^{-\textbf{j}\left(c_2 w_2^{2}+e_2 w_2\right)}}{(\textbf{j}b_2)^{\mu+1}}
\end{equation}
where
$g(v_1,v_2)=e^{-\textbf{i}\left(a_1v_1^{2}+d_1v_1\right)}f(v_1,v_2)e^{-\textbf{j}\left(a_2v_2^{2}+d_2v_2\right)}.$    
\end{enumerate}
\end{thm}
\begin{proof}
\begin{enumerate}
\item Let $f\in L^1_{\mu}(\mathbb{R}^2,\mathbb{H})$. Since
$E_\mu(-\textbf{i}w,v)\leq 1,$
the mapping $w \longmapsto E_{\mu}(\textbf{i}w,v)$ is continuous on $\mathbb{R}^2$, hence the mapping $w \longmapsto \mathbb{D}^{a,b,c}_{d,e,\mu}[f](w)$
is also continuous for every $w \in \mathbb{R}^2$, we have
\begin{eqnarray*}
\left| \mathbb{D}^{a,b,c}_{d,e,\mu}[f](w) \right|
&=&\Big| c_{\mu}^2 
\int_{\mathbb{R}^2} \Psi^{a_1, b_1, c_1}_{d_1, e_1, \mu}(w_1, v_1) f(v_1, v_2) \Psi^{a_2, b_2, c_2}_{d_2, e_2, \mu}(w_2, v_2)~ d\lambda_\mu(v)
\Big| \\
&\leq&\frac{c_{\mu}^2}{|b_1|^{\mu+1}|b_2|^{\mu+1}}
\int_{\mathbb{R}^2}
\left|  f(v_1, v_2)\right|~  d\lambda_\mu(v) \\
&\leq & \frac{c_{\mu}^2}{|b_1 b_2|^{\mu+1}}
\|f\|_{1,\mu}.
\end{eqnarray*}
\item Let $f \in \mathcal{S}(\mathbb{R}^2,\mathbb{H})$, Then we have
\begin{eqnarray*}
\mathbb{D}^{a,b,c}_{d,e,\mu}[f](w)
&=& c_{\mu}^2
\int_{\mathbb{R}^2} \Psi^{a_1, b_1, c_1}_{d_1, e_1, \mu}(w_1, v_1) f(v_1, v_2) \Psi^{a_2, b_2, c_2}_{d_2, e_2, \mu}(w_2, v_2) d\lambda_\mu(v)  \\
&=&c_{\mu}^2 \int_{\mathbb{R}^2} \frac{E_{\mu}(-\textbf{i}kw_1/b_1, v_1)}{(\textbf{i}b_1)^{\mu+1}} e^{-\textbf{i}(a_1 v_1^2 + c_1 k w_1^2 + d_1 v_1 + e_1 k w_1)}  \\
&\times&f(v_1, v_2) \frac{E_{\mu}(-\textbf{j}kw_2/b_2, v_2)}{(\textbf{j}b_2)^{\mu+1}} e^{-\textbf{j}(a_2 v_2^2 + c_2 k w_2^2 + d_2 v_2 + e_2 k w_2)} d\lambda_\mu(v)  \\
&=&\frac{e^{-\textbf{i}\left(c_1w_1^2+e_1w_1\right)}}{(\textbf{i}b_1)^{\mu+1}}
\Big[c_{\mu}^2
\int_{\mathbb{R}^2}
E_{\mu}\left(-	\textbf{i}w_1/b_1,v_1\right) e^{\textbf{i}\left(a_1v_1^{2}+d_1v_1\right)} f(v_1,v_2) 
  \\
&\times&  e^{-\textbf{j}\left(a_2v_2^{2}+d_2 v_2\right)} E_{\mu}\left(-\textbf{j}w_2/b_2,v_2\right)  d\lambda_\mu(v)  \Big] \frac{ e^{-\textbf{j}\left(c_2 w_2^{2}+e_2 w_2\right)}}{(\textbf{j}b_2)^{\mu+1}}   \\
&=& \frac{ e^{-\textbf{i}\left(c_1w_1^2+e_1w_1\right)} }{(\textbf{i}b_1)^{\mu+1}}
\mathcal{D}^{\mathbb{H}} _{\mu}(g)\!\left(\frac{w_1}{b_1},\frac{w_2}{b_2}\right) \frac{e^{-\textbf{j}\left(c_2 w_2^{2}+e_2 w_2\right)}}{(\textbf{j}b_2)^{\mu+1}},
\end{eqnarray*}
where
$g(v_1,v_2)=e^{-\textbf{i}\left(a_1v_1^{2}+d_1v_1\right)} f(v_1,v_2) e^{-\textbf{j}\left(a_2v_2^{2}+d_2v_2\right)}. $
\end{enumerate}
\end{proof}
\begin{prop}  
[Riemann Lebesgue lemma for the QQPDT] 
For every $f \in L^1_\mu(\mathbb{R}^2,\mathbb{H})$, the
QQPDT $\mathbb{D}^{a,b,c}_{d,e,\mu}[f] \in C_0(\mathbb{R}^2,\mathbb{H})$ and satisfies
\begin{equation}
\left\|
\mathbb{D}^{a,b,c}_{d,e,\mu}[f]
\right\|_{\infty,\mu}
\le
\frac{c_{\mu}^2}{|b_1 b_2|^{\mu+1}} \, \|f\|_{1,\mu}.
\end{equation}
\end{prop}
\begin{proof}
Let $f \in  L^1_\mu(\mathbb{R}^2, \mathbb{H})$. From Theorem \ref{th3.1}, we know that the transform
$\mathbb{D}^{a,b,c}_{d,e,\mu}[f]$ is continuous on $\mathbb{R}^2$ and
bounded by
$\frac{c_{\mu}^2}{|b_1 b_2|^{\mu+1}} \, \|f\|_{1,\mu}$.
To show that it converges at infinity, let
$f \in \mathcal{S}(\mathbb{R}^2,\mathbb{H})$. From theorem \ref{th3.1}, we have
\begin{equation*}
\mathbb{D}^{a,b,c}_{d,e,\mu}[f](w)
= \frac{ e^{-\textbf{i}\left(c_1w_1^2+e_1w_1\right)} }{(\textbf{i}b_1)^{\mu+1}}
\mathcal{D}^{\mathbb{H}} _{\mu}(g)\!\left(\frac{w_1}{b_1},\frac{w_2}{b_2}\right) \frac{e^{-\textbf{j}\left(c_2 w_2^{2}+e_2 w_2\right)}}{(\textbf{j}b_2)^{\mu+1}},
\end{equation*}
where
\begin{eqnarray*}
g(v_1,v_2)=e^{-\textbf{i}\left(a_1v_1^{2}+d_1v_1\right)}f(v_1,v_2)e^{-\textbf{j}\left(a_2v_2^{2}+d_2v_2\right)}.
\end{eqnarray*}
Since $f \in \mathcal{S}(\mathbb{R}^2,\mathbb{H})$, it follows that
$g \in \mathcal{S}(\mathbb{R}^2,\mathbb{H})$. By Riemann--Lebesgue lemma
for the quaternion Dunkl transform, we have
$\mathcal{D}^{\mathbb{H}}_{\mu}(g) \in C_0(\mathbb{R}^2,\mathbb{H})$.
The expression $\mathcal{D}_{\mu}(g)(w/b)$, after multiplication by the bounded continuous functions
$e^{-\textbf{i}\left(c_1w_1^{2}+e_1 w_1\right)}$
and $e^{-\textbf{j}\left(c_2w_2^{2}+e_2 w_2\right)}$, continues to vanish at infinity.
Hence,$\mathbb{D}^{a,b,c}_{d,e,\mu}[f] \in C_0(\mathbb{R}^2,\mathbb{H}).$\\
For $f \in \mathcal{S}(\mathbb{R}^2,\mathbb{H})$.
Now, in general $f \in L^1_\mu(\mathbb{R}^2,\mathbb{H})$ and  let
$\{f_n\} \subset  \mathcal{S}(\mathbb{R}^2,\mathbb{H})$ be a sequence such that
\begin{eqnarray*}
\|f_n - f\|_{1,\mu} \longrightarrow 0
\quad \text{as } n \to \infty.
\end{eqnarray*}
Then we have
\begin{eqnarray*}
\left\|
\mathbb{D}^{a,b,c}_{d,e,\mu}[f_n]
-\mathbb{D}^{a,b,c}_{d,e,\mu}[f]
\right\|_{\infty, \mu}
\le\frac{c_{\mu}^2}{|b_1 b_2|^{\mu+1}} \, \|f_n-f\|_{1,\mu}
\longrightarrow 0 .
\end{eqnarray*}
which completes the proof.
\end{proof}   

\begin{thm}[Inversion property of QQPDT] Let f be a function belonging to $L^1_\mu(\mathbb{R}^2,\mathbb{H})$ and $\mathbb \mathbb {D}^{a, b, c}_{d, e, {\mu}} [f] \in L^1_\mu(\mathbb{R}^2,\mathbb{H})$ then the inverse of the QQPDT is given by
\begin{eqnarray} 
\mathbb D^{-c, -b, -a}_{-e, -d, {\mu}} \left( \mathbb {D}^{a, b, c}_{d, e, {\mu}} [f(u)](w) \right)(v_1,v_2)
= f(v_1, v_2) 
\end{eqnarray}
Equivalently 
\begin{eqnarray}
  c_{\mu}^2 \int_{\mathbb{R}^2} \overline{\Psi^{a_1, b_1, c_1}_{d_1, e_1, \mu}(w_1, v_1)}  \, \mathbb{D}^{a,b,c}_{d,e,\mu}[f](w_1,w_2) \, \overline{\Psi^{a_2, b_2, c_2}_{d_2, e_2, \mu}(w_2, w_2)} \notag~ d\lambda_\mu(w)=f(v_1, v_2), 
\end{eqnarray}
\begin{proof} Let $f \in L^1_\mu(\mathbb{R}^2,\mathbb{H})$ and $ \mathbb {D}^{a, b, c}_{d, e, {\mu}} [f] \in L^1_\mu(\mathbb{R}^2,\mathbb{H})$ then, we have
\begin{eqnarray*}
&&\mathbb D^{-c, -b, -a}_{-e, -d, {\mu}}  \left( \mathbb {D}^{a, b, c}_{d, e, {\mu}} [f(u)](w) \right)(v) \\
&=& c_{\mu}^2  \int_{\mathbb{R}^2} \Psi^{-c_1, -b_1, -a_1}_{-e_1, -d_1, {\mu}}(v_1, w_1) 
\left( \mathbb {D}^{a, b, c}_{d, e, {\mu}}[f(u)](w_1, w_2) \right) 
 \Psi^{-c_2, -b_2, -a_2}_{-e_2, -d_2, {\mu}}(v_2, w_2) \, d\lambda_\mu(w)
\\
&=& c_{\mu}^2
\int_{\mathbb{R}^2} \Psi^{-c_1, -b_1, -a_1}_{-e_1, -d_1, {\mu}}(v_1, w_1) \int_{\mathbb{R}^2} \Psi^{a_1, b_1, c_1}_{d_1, e_1, {\mu}}(w_1, u_1)  f(u_1, u_2) 
\Psi^{a_2, b_2, c_2}_{d_2, e_2, {\mu}}(w_2, u_2)  \,d\lambda_\mu(u) \\
&\times&  \Psi^{-c_2, -b_2, -a_2}_{-e_2, -d_2, {\mu}}(v_2, w_2) \, d\lambda_\mu(w) 
\end{eqnarray*}
\\
\begin{eqnarray*}
&=&\frac{c_{\mu}^2 c_{\mu}^2}{|b_1|^{2{\mu}+2} |b_2|^{2{\mu}+2}} \int_{\mathbb{R}^2} E_{\mu}(\textbf{i} w_1/b_1, v_1) e^{\textbf{i}({c_1 w_1^2} + a_1 v_1^2 + {e_1 w_1} + d_1 v_1)} \Big[ \int_{\mathbb{R}^2} E_{\mu}(-\textbf{i} w_1/b_1, u_1) \\  
 &\times& \left. e^{-\textbf{i}(a_1 u_1^2 + {c_1 w_1^2} + d_1 u_1 + {e_1 w_1})} f(u_1, u_2)  ~E_{\mu}(-\textbf{j} w_2/b_2, u_2) ~ e^{-\textbf{j}(a_2 u_2^2 + {c_2 w_2^2} + d_2 u_2 + {e_2 w_2})} \right] \\
&\times &  E_{\mu}(\textbf{j} w_2/b_2, v_2) ~e^{\textbf{j}({c_2 w_2^2} + a_2 v_2^2 + {e_2 w_2} + d_2 v_2)}~ d\lambda_\mu(u)~ d\lambda_\mu(w)  \\
&=&\frac{c_{\mu}^2 c_{\mu}^2 e^{\textbf{i}(a_1 v_1^2 + d_1 v_1)}}{|b_1 b_2|^{2{\mu}+2}}  \int_{\mathbb{R}^2} E_{\mu}(\textbf{i} w_1/b_1, v_1) E_{\mu}(-\textbf{i} w_1/b_1, u_1)  e^{-\textbf{i}(a_1 u_1^2 + d_1 u_1)} \int_{\mathbb{R}^2} f(u_1, u_2)  \\
&\times&  E_{\mu}(-\textbf{j} w_2/b_2, u_2) E_{\mu}(\textbf{j} w_2/b_2, v_2)    e^{-\textbf{j}(a_2 u_2^2 + d_2 u_2)} e^{\textbf{j}(a_2 v_2^2 + d_2 v_2)} ~ d\lambda_\mu(u)~ d\lambda_\mu(w) 
\end{eqnarray*}
Taking \,$ z_1=w_1/b_1$ and $\, z_2=w_2/b_2 $, we get 
\begin{eqnarray*}
&&\mathbb D^{-c, -b, -a}_{-e, -d, {\mu}} \left( \mathbb {D}^{a, b, c}_{d, e, {\mu}} [f(u)](w) \right)(v) \\
&=& \frac{c_{\mu}^2}{{|b_1 b_2|^{2{\mu}+2} }} e^{\textbf{i}(a_1 v_1^2 + d_1 v_1)} \Big \{ c_{\mu}^2 \int_{\mathbb{R}^2}  E_{\mu}(\textbf{i} z_1, v_1) \Big [ \int_{\mathbb{R}^2}  E_{\mu}(-\textbf{i} z_1, u_1) e^{-\textbf{i}(a_1 u_1^2 + d_1 u_1)} f(u_1, u_2)  \\ 
&\times&  E_{\mu}(-\textbf{j} z_2, u_2) e^{-\textbf{j}(a_2 u_2^2 + d_2 u_2)}    \left. E_{\mu}(\textbf{j} z_2, v_2) \right. {|b_1|^{2{\mu}+2}} {|b_2|^{2{\mu}+2}}  d\lambda_\mu(z) \Big ] d\lambda_\mu(u)\Big\} e^{\textbf{j}(a_2 v_2^2 + d_2 v_2)} \\
&=& e^{\textbf{i}(a_1 v_1^2 + d_1 v_1)}  \Big \{ c_{\mu}^2 \int_{\mathbb{R}^2} E_{\mu}(\textbf{i} z_1, v_1) \Big[ c_{\mu}^2 \int_{\mathbb{R}^2} E_{\mu}(-\textbf{i} z_1, u_1)   g(u_1, u_2) E_{\mu}(-\textbf{j} z_2, u_2) d\lambda_\mu(u) \Big ] \\
&\times&  E_{\mu}(\textbf{j} z_2, v_2)  d\lambda_\mu(z) \Big\}~ e^{\textbf{j}(a_2 v_2^2 + d_2 v_2)}.
\end{eqnarray*} 
Here $g(u_1, u_2) = e^{-\textbf{i}(a_1 u_1^2 + d_1 u_1)} f(u_1, u_2) e^{-\textbf{j}(a_2 u_2^2 + d_2 u_2)}$ and by using \eqref{eq2.10}, we get
\begin{eqnarray*}
&&\mathbb D^{-c, -b, -a}_{-e, -d, {\mu}} \left( \mathbb {D}^{a, b, c}_{d, e, {\mu}} [f(u)](w) \right)(v)\\
&=& e^{\textbf{i}(a_1 v_1^2 + d_1 v_1)} \Big \{ c_{\mu}^2 \int_{\mathbb{R}^2} E_{\mu}(\textbf{i} z_1, v_1) ~ \mathcal{D}^{\mathbb{H}}_{\mu}g(u_1,u_2) ~E_{\mu}(\textbf{j} z_2, v_2) ~d\lambda_\mu(z) \Big \} ~ e^{\textbf{j}(a_2 v_2^2 + d_2 v_2)} \\
& = & e^{\textbf{i}(a_1 v_1^2 + d_1 v_1)} \mathcal{D}_{\mu}^{-1}(\mathcal{D}^{\mathbb{H}}_{\mu} g)(v_1, v_2) ~  e^{\textbf{j}(a_2 v_2^2 + d_2 v_2)} 
\\
&= &e^{\textbf{i}(a_1 v_1^2 + d_1 v_1)}  \left( e^{-\textbf{i}(a_1 v_1^2 + d_1 v_1)} f(v_1, v_2) e^{-\textbf{j}(a_2 v_2^2 + d_2 v_2)} \right) e^{\textbf{j}(a_2 v_2^2 + d_2 v_2)}  \\
&=& f(v_1, v_2).
\end{eqnarray*}
Hence the theorem is proved.
\end{proof}
\end{thm} 
\begin{thm}[Parseval's identity for two-sided QQPDT] Let $ f, g \in L^2_{\mu}(\mathbb{R}^2,\mathbb{H})$, then we have
\begin{eqnarray*}
\int_{\mathbb{R}^2}f(v_1, v_2)~ \overline{g(v_1, v_2)} ~d\lambda_\mu(v)=\int_{\mathbb{R}^2} \mathbb D^{a,b,c}_{d,e,{\mu}}[f](w_1, w_2) ~\overline{\mathbb D^{a,b,c}_{d,e,{\mu}}[g] (w_1, w_2)}~
d\lambda_\mu(w).
\end{eqnarray*}
\begin{proof}
Let $f,g \in L^2_{\mu}(\mathbb{R}^2,\mathbb{H})$ and from \eqref{eq3.2}, we have
\begin{eqnarray*}
\mathbb D^{a,b,c}_{d,e,{\mu}} [f(v_1, v_2)](w_1, w_2) &=& \!  \! \frac{e^{-\textbf{i}(c_1 w_1^2 + e_1 w_1)} }{(\textbf{i}b_1)^{{\mu}+1}}  \mathcal{D}^{\mathbb{H}}_{\mu} [f_1] \left(\frac{w_1}{b_1}, \frac{w_2}{b_2}\right) \frac{e^{-\textbf{j}(c_2 w_2^2 + e_2 w_2)} }{(\textbf{j}b_2)^{{\mu}+1}}, \\
\textit{and}  ~ \overline{\mathbb D^{a,b,c}_{d,e,{\mu}}[g(v_1,v_2)](w_1, w_2)} &=& \!  \! \frac{e^{\textbf{i}(c_1 w_1^2 + e_1 w_1)} }{(-\textbf{i}b_1)^{{\mu}+1}} ~ \overline{\mathcal{D}^{\mathbb{H}}_{\mu} [g_1] \left(\frac{w_1}{b_1}, \frac{w_2}{b_2}\right)} ~ 
\frac{e^{\textbf{j}(c_2 w_2^2 + e_2 w_2)} }{(-\textbf{j}b_2)^{{\mu}+1}},\\
\textit{where} \quad
f_1(v_1, v_2) &=& e^{-\textbf{i}(a_1 v_1^2 + d_1 v_1)} f(v_1, v_2) e^{-\textbf{j}(a_2 v_2^2 + d_2 v_2)} \\ \textit{and} \quad
g_1(v_1, v_2) &=& e^{-\textbf{i}(a_1 v_1^2 + d_1 v_1)} g(v_1, v_2) e^{-\textbf{j}(a_2 v_2^2 + d_2 v_2).}
\end{eqnarray*}
Now,
\begin{eqnarray*}
&&\int_{\mathbb{R}^2}\mathbb D^{a,b,c}_{d,e,{\mu}} [f(v_1, v_2)](w_1,w_2)
\overline{\mathbb D^{a,b,c}_{d,e,{\mu}}[g(v_1,v_2)](w_1, w_2)} ~d\lambda_\mu(w) \\
&=& \frac{1}{|b_1|^{2{\mu}+2} |b_2|^{2{\mu}+2}}\int_{\mathbb{R}^2} \mathcal{D}^{\mathbb{H}} _{\mu} [f_1] \left(\frac{w_1}{b_1}, \frac{w_2}{b_2}\right) \overline{\mathcal{D} ^{\mathbb{H}}_{\mu} [g_1] \left(\frac{w_1}{b_1}, \frac{w_2}{b_2}\right)} d\lambda_\mu(w).
\end{eqnarray*}
 Put $z_1 = \frac{w_1}{b_1}$,  $z_2 = \frac{w_2}{b_2}$, and Using theorem \eqref{thm2.2}, we get
\begin{eqnarray*}
&\,& \int_{\mathbb{R}^2} \mathbb D^{a,b,c}_{d,e,{\mu}}[f(v_1,v_2)](w_1, w_2) \overline{\mathbb D^{a,b,c}_{d,e,{\mu}}[g(v_1,v_2)] (w_1, w_2)} ~d\lambda_\mu(w) \\
&=& \frac{1}{|b_1 b_2|^{2{\mu}+2}} \int_{\mathbb{R}^2} \mathcal{D} ^{\mathbb{H}}_{\mu}[f_1](z_1, z_2)~ \overline{\mathcal{D} ^{\mathbb{H}}_{\mu}[g_1](z_1, z_2)} ~|b_1|^{2{\mu}+1} {|b_2|^{2{\mu}+1}} b_1 b_2 \, d\lambda_\mu(z) \\
&=&\! \int_{\mathbb{R}^2} \mathcal{D} ^{\mathbb{H}}_{\mu}[f_1(v_1,v_2)](z_1, z_2) 
~\overline{\mathcal{D} ^{\mathbb{H}}_{\mu}[g_1(v_1,v_2)](z_1, z_2)} ~
d\lambda_\mu(z).\\
&=&\int_{\mathbb{R}^2} f_1(v_1, v_2) \overline{g_1(v_1, v_2)} ~d\lambda_\mu(v) \\
&=& e^{-\textbf{i}(a_1 v_1^2 + d_1 v_1)} f(v_1, v_2) e^{-\textbf{j}(a_2 v_2^2 + d_2 v_2)} \overline{e^{-\textbf{i}(a_1 v_1^2 + d_1 v_1)} ~g(v_1, v_2) e^{-\textbf{j}(a_2 v_2^2 + d_2 v_2)}} ~d\lambda_\mu(v)\\
&=& e^{-\textbf{i}(a_1 v_1^2 + d_1 v_1)} f(v_1, v_2) e^{-\textbf{j}(a_2 v_2^2 + d_2 v_2)} ~e^{\textbf{j}(a_2 v_2^2 + d_2 v_2)}~\overline{g(v_1, v_2)}  ~e^{\textbf{i}(a_1 v_1^2 + d_1 v_1)}~d\lambda_\mu(v)\\
&=& \int_{\mathbb{R}^2} f(v_1, v_2) \overline{g(v_1, v_2)} ~d\lambda_\mu(v).
\end{eqnarray*}
\end{proof}
\end{thm} 
In particular, if $f=g$, we get the Plancherel's formula 
\begin{eqnarray}
\left\|
\mathbb D^{a,b,c}_{d,e,\mu}[f]
\right\|_{L^2_\mu(\mathbb{R}^2,\mathbb{H})}
=\|f\|_{L^2_\mu(\mathbb{R}^2,\mathbb{H})}. 
\end{eqnarray}

\section{Heisenberg type Uncertainty Principle for two-sided QQPDT}
In this section, we develop a Heisenberg type uncertainty principle for the two-sided QQPDT. It establishes a lower bound on the product of the spreads of a function and its transform, showing that they cannot both be simultaneously well localized.
\begin{thm}
[Heisenberg type uncertainty principle for quaternion Dunkl transform  \cite{UPQDT}]
Let $f \in L^2_\mu(\mathbb{R}^2,\mathbb{H})$. Then we have
\begin{eqnarray}
&&\left(\int_{\mathbb{R}^2} |v|^{2} |f(v)|^{2}d\lambda_\mu(v)
\right)\left(\int_{\mathbb{R}^2} |w|^{2}
\left|\mathcal{D}^{\mathbb{H}}_{\mu}[f](w)\right|^{2} d\lambda_\mu(w)\right)\notag \\
&&\quad \quad \quad \quad \ge \left(2\mu+1\right)^{2}\left(\int_{\mathbb{R}^2} |f(v)|^{2} d\lambda_\mu(v)\right)^{2}.
\end{eqnarray}
\end{thm}
Now we are going to establish the Heisenberg-type uncertainty principle for the two-sided QQPDT.
\begin{thm}[Heisenberg type inequality for quaternion  quadratic-phase Dunkl transform]
Let $f \in L^2_\mu(\mathbb{R}^2,\mathbb{H}),$  we have
\begin{eqnarray}
&& \left(\int_{\mathbb{R}^2} |v|^{2} |f(v)|^{2}  d\lambda(v) 
\right)\left(\int_{\mathbb{R}^2} |w|^{2}
\left|\mathbb{D}^{a,b,c}_{d,e,\mu}[f](w)\right|^{2} d\lambda_\mu(w)\right) \notag \\
&&\qquad \ge |b_1 b_2|^{2\mu+2}\left(2\mu+1\right)^{2} 
\left(\int_{\mathbb{R}^2} |f(v)|^{2} d\lambda_\mu(v)
\right)^{2}. \label{eq4.2}
\end{eqnarray}
\end{thm}
\begin{proof}
From the relation \eqref{eq3.2} we have
\begin{eqnarray}\label{eq4.3}
\mathbb{D}^{a,b,c}_{d,e,\mu}[f](w)
= \frac{ e^{-\textbf{i}\left(c_1w_1^2+e_1w_1\right)} }{(\textbf{i}b_1)^{\mu+1}}
\mathcal{D}^{\mathbb{H}} _{\mu}(g)\!\left(\frac{w_1}{b_1},\frac{w_2}{b_2}\right) \frac{e^{-\textbf{j}\left(c_2 w_2^{2}+e_2 w_2\right)}}{(\textbf{j}b_2)^{\mu+1}},
\end{eqnarray}
where
\begin{eqnarray*}
g(v_1,v_2)=e^{-\textbf{i}\left(a_1v_1^{2}+d_1v_1\right)} f(v_1,v_2) e^{-\textbf{j}\left(a_2v_2^{2}+d_2v_2\right)}.
\end{eqnarray*}
Now, apply the Heisenberg inequality for the quaternion Dunkl transform to the function $g$, we obtain
\begin{eqnarray}\label{eq4.4}
\left(
\int_{\mathbb{R}^2} |v|^{2} |g(v)|^{2} d\lambda_\mu(v)
\right)&\times&\left(\int_{\mathbb{R}^2} |w|^{2}
\left|\mathcal{D}_{\mu}[g](w)\right|^{2}
d\lambda_\mu(w)
\right) \notag\\
&\ge& \left(2\mu+1\right)^{2}
\left(
\int_{\mathbb{R}^2} |g(v)|^{2} d\lambda_\mu(v)
\right)^{2}.
\end{eqnarray}
Since $|g(v)| = |f(v)|$, then $1^{st}$ integral becomes
\begin{eqnarray*}
\int_{\mathbb{R}^2} |g(v)|^{2} d\lambda(v)
&=&\int_{\mathbb{R}^2} |f(v)|^{2} d\lambda(v),
\end{eqnarray*}
which implies
\begin{eqnarray*}
\int_{\mathbb{R}^2} |v|^{2} |g(v)|^{2} d\lambda(v)
&=&\int_{\mathbb{R}^2} |v|^{2} |f(v)|^{2} d\lambda(v) .
\end{eqnarray*}
Now substitution $w = \left(\frac{u_1}{b_1},\frac{u_2}{b_2}\right)$ in the second integral, we obtain
\begin{eqnarray}
I_2=\int_{\mathbb{R}^2}\left(\frac{u_1^2}{b_1^2} + \frac{u_2^2}{b_2^2}
\right)\left|\mathcal{D}^{\mathbb{H}}_\mu[g]\!\left(\frac{u_1}{b_1},\frac{u_2}{b_2}\right)\right|^2\frac{d\lambda_\mu(u)}{|b_1|^{2\mu+2}|b_2|^{2\mu+2}}. \label{eq4.5}
\end{eqnarray}
Taking absolute values in equation \eqref{eq4.3}, we get
\begin{eqnarray*}
\left|
\mathbb{D}^{a,b,c}_{d,e,\mu}[f](w)
\right|
=\frac{1}{|b_1|^{\mu+1} \, |b_2|^{\mu+1}}
\left|
\mathcal{D} ^{\mathbb{H}}_{\mu}\!\left[g\right]\!\left(\frac{w_1}{b_1},\frac{w_2}{b_2}\right)
\right| .
\end{eqnarray*}
Substituting this into the equation \eqref{eq4.5}, we obtain
\begin{eqnarray}
I_2=\int_{\mathbb{R}^2}\left(
\frac{u_1^2}{b_1^2} + \frac{u_2^2}{b_2^2}\right)\left|\mathbb{D}^{a,b,c}_{d,e,\mu}[f](w)\right|^2
d\lambda_\mu(u).
\end{eqnarray}
This implies
\begin{eqnarray}
I_2=|b_1|^{2\mu+2}|b_2|^{2\mu+2} \int_{\mathbb{R}^2}\left({
{w_1}^2+{w_2}^2}\right)\left|\mathbb{D}^{a,b,c}_{d,e,\mu}[f](w)\right|^2
d\lambda_\mu(w).
\end{eqnarray}
Now equation \eqref{eq4.4} becomes 
\begin{eqnarray*}
&& \left(\int_{\mathbb{R}^2} |v|^{2} |f(v)|^{2}  d\lambda_\mu(v) 
\right)  \left(\int_{\mathbb{R}^2} |w|^{2}
\left|\mathbb{D}^{a,b,c}_{d,e,\mu}[f](w)\right|^{2} d\lambda_\mu(w) \right) \notag \\
&&\qquad \ge |b_1 b_2|^{2\mu+2} \left(2\mu+1\right)^{2}
\left(\int_{\mathbb{R}^2} |f(v)|^{2} d\lambda_\mu(v) 
\right)^{2}.
\end{eqnarray*}
\end{proof}
Which completes the proof.

\section{Donoho--Stark Uncertainty Principle}
In this section, we establish a Donoho-Stark-type uncertainty principle \cite{chen2015, Sahbani, Dar2023} for the QDT and QQPDT. It characterizes the trade off between the concentration of a function and its transform on measurable sets. we first introduce the notions of time-limitation and frequency-limitation in both QDT and QQPDT setting.
To do so, we first introduce the notion of $\varepsilon$-concentrate of quaternion valued signal on a measurable set. 

\begin{defn}[$\varepsilon_\mathcal{M}$-concentrated on $\mathcal{M}$ for QDT and QQPDT] 
Let $\mathcal{M}$ be a measurable subset of $\mathbb{R}^{2}$.
A function
$f\in L_{\mu}^{2}(\mathbb{R}^{2},\mathbb{H})$
is said to be $\varepsilon_\mathcal{M}$-concentrated on $\mathcal{M}$ if
\begin{equation}
\|f-\chi_{\mathcal{M}}f\|_{2,\mu}
\leq
\varepsilon_\mathcal{M}\|f\|_{2,\mu},
\label{eps-concentration}
\end{equation}
where the characteristic function of $\mathcal{M}$ is defined by
\[
\chi_{\mathcal{M}}(v_1,v_2)
=\begin{cases}1, & \text{if } (v_1,v_2)\in\mathcal{M},\\[2mm]
0, & \text{otherwise}.
\end{cases}
\]
\end{defn}
Since \quad $f-\chi_{\mathcal{M}}f
=\chi_{\mathcal{M}^{c}}f,$
it follows that
$$
\|f-\chi_{\mathcal{M}}f\|_{2,\mu}^{2}
=
\|\chi_{\mathcal{M}^{c}}f\|_{2,\mu}^{2}
\nonumber\\
=\int_{\mathcal{M}^{c}}
|f(v_1,v_2)|^{2}
d\lambda_\mu(v).
\label{energy-outside}
$$

Therefore, condition \eqref{eps-concentration} is equivalent to

\begin{eqnarray}
\int_{\mathcal{M}^{c}}
|f(v_1,v_2)|^{2}
d\lambda_\mu(v)
\leq
\varepsilon_\mathcal{M}^{2} \notag
\|f\|_{2,\mu}^{2},
\end{eqnarray}
\begin{eqnarray}
\text{which implies} \quad
\left(
\int_{\mathbb R^2\setminus \mathcal M}
|f(v_1,v_2)|^{2}\,d\lambda_\mu(v)
\right)^{\frac12}
\le
\varepsilon_\mathcal{M} \|f\|_{2,\mu} \label{eq4.33}
\end{eqnarray}

Hence, an $\varepsilon_\mathcal{M}$-concentrated function has at most an
$\varepsilon_\mathcal{M}$-fraction of its total energy outside $\mathcal{M}$.
Equivalently, most of its energy is concentrated in $\mathcal{M}$.

In particular, when $\varepsilon_\mathcal{M}=0$, we obtain
\[
f=\chi_{\mathcal{M}}f,
\]
which implies that $f$ vanishes outside $\mathcal{M}$ and is exactly
supported on $\mathcal{M}$.

Similarly, we say that the QDT $\mathcal{D}^{\mathbb{H}}_\mu [f]$ is $\varepsilon_\mathcal M$-concentrated on a measurable set
$\mathcal M\subseteq \mathbb R^2$, if
\[
\left(
\int_{\mathbb R^2\setminus \mathcal M}
\left|
\mathcal{D}^{\mathbb{H}}_\mu [f](w_1,w_2)
\right|^2
d\lambda_\mu(v)
\right)^{\frac12}
\le
\varepsilon_\mathcal{M}
\left\|
 \mathcal{D}^{\mathbb{H}}_\mu [f]
\right\|_{2,\mu}.
\]
And we say that the QQPDT $D^{a,b,c}_{d,e,\mu}[f]$ is said to be $\varepsilon_\mathcal M$-concentrated on a measurable set
$\mathcal M\subseteq \mathbb R^2$, if
\begin{eqnarray*}
\left(
\int_{\mathbb R^2\setminus \mathcal M}
|\mathbb{D}^{a,b,c}_{d,e,\mu}[f](w_1,w_2)|^2\,d\lambda_\mu(v)
\right)^{\frac12}
\le
\varepsilon_\mathcal M\|\mathbb{D}^{a,b,c}_{d,e,\mu}[f](w_1,w_2)\|_{2,\mu}
\end{eqnarray*}

Similarly, we can define $\varepsilon_{\mathcal N}$-concentration on the set $\mathcal N$ for QDT and QQPFBT.
\begin{thm}
[Donoho-Stark uncertainty principle for QDT] 
Let $f\in L^2_\mu(\mathbb R_+,\mathbb H)$ and 
assume that $f$ is $\varepsilon_{\mathcal M}$-concentrated on
$\mathcal M\subseteq \mathbb R^2$ and that
$\mathcal{D}^{\mathbb{H}}_\mu [f]$
is $\varepsilon_{\mathcal N}$-concentrated on
$\mathcal N\subseteq \mathbb R^2$.
Then
\begin{eqnarray}
|\mathcal M|\,|\mathcal N|
\ge
\frac{1}{c_\mu^4}
\Bigl(
1-\varepsilon_{\mathcal M}
-\varepsilon_{\mathcal N}
\Bigr)^2. \label{leema5.1}
\end{eqnarray}
where $|\mathcal M|$ and $|\mathcal N|$ denote the measures of the sets
$\mathcal M$ and $\mathcal N$, respectively.
\end{thm} 
\begin{proof}
Without loss of generality, we may assume that $\mathcal M$ and $\mathcal N$ have finite measure.

Define the time-limiting operator
\[
P_{\mathcal M}f
=
\chi_{\mathcal M}f,
\]
where $\chi_{\mathcal M}$ denotes the characteristic function of $\mathcal M$.

Further, define the frequency-limiting operator
\[
Q_{\mathcal N}f
=
(\mathcal{D}_\mu^{\mathbb H})^{-1}
\Bigl[
\chi_{\mathcal N}
\,\mathcal{D}_\mu^{\mathbb H}[f]
\Bigr].
\]

Since the quaternion Dunkl transform $\mathcal{D}_\mu^{\mathbb H}$ is unitary on
$L_\mu^2(\mathbb{R}^2,\mathbb H)$,
the operators $P_{\mathcal M}$ and $Q_{\mathcal N}$ are orthogonal projections. Consequently, $\|P_{\mathcal M} \|\leq 1.$

Since $f$ is $\varepsilon_{\mathcal M}$-concentrated on $\mathcal M$, we have
\begin{equation}
\|f-P_{\mathcal M}f\|_{2,\mu}
\leq
\varepsilon_{\mathcal M}\|f\|_{2,\mu}.
\label{eq1}
\end{equation}

similarly, $f$ is
$\varepsilon_{\mathcal N}$-concentrated on $\mathcal N$,
\begin{equation}
\|f-Q_{\mathcal N}f\|_{2,\mu}
\leq
\varepsilon_{\mathcal N}\|f\|_{2,\mu}.
\label{eq2}
\end{equation}

Using the triangle inequality, we obtain
\begin{eqnarray}
\|f-P_{\mathcal M}Q_{\mathcal N}f\|_{2,\mu}
&=&
\|f-P_{\mathcal M}f
+
P_{\mathcal M}f
-
P_{\mathcal M}Q_{\mathcal N}f\|_{2,\mu} \notag\\
&\leq&
\|f-P_{\mathcal M}f\|_{2,\mu}
+
\|P_{\mathcal M}(f-Q_{\mathcal N}f)\|_{2,\mu}.\label{eq5.66}
\end{eqnarray}

Since $\|P_{\mathcal M}\|\leq 1$, it follows that
\[
\|P_{\mathcal M}(f-Q_{\mathcal N}f)\|_{2,\mu}
\leq
\|f-Q_{\mathcal N}f\|_{2,\mu}.
\]

Combining this with \eqref{eq1} and \eqref{eq2}, then \eqref{eq5.66} becomes
\begin{eqnarray}
\|f-P_{\mathcal M}Q_{\mathcal N}f\|_{2,\mu}
\leq
(\varepsilon_{\mathcal M}
+
\varepsilon_{\mathcal N})
\|f\|_{2,\mu}. \label{eq5.5}
\end{eqnarray}
Again applying triangle inequality, we get
\begin{eqnarray}
\|f-P_{\mathcal M}Q_{\mathcal N}f\|_{2,\mu} \geq \|f\|_{2,\mu}- \|P_{\mathcal M}Q_{\mathcal N}f\|_{2,\mu}. \label{eq5.6}
\end{eqnarray}
From equation \eqref{eq5.5} and \eqref{eq5.6}, we get
\begin{eqnarray*} \|f\|_{2,\mu}- \|P_{\mathcal M}Q_{\mathcal N}f\|_{2,\mu} &\leq& (\varepsilon_{\mathcal M}
+\varepsilon_{\mathcal N})
\|f\|_{2,\mu} \\
\implies 
\|P_{\mathcal M}Q_{\mathcal N}f\|_{2,\mu}
&\geq&
\left(
1-\varepsilon_{\mathcal M}
-\varepsilon_{\mathcal N}
\right)
\|f\|_{2,\mu}.
\end{eqnarray*}

Now dividing by $\|f\|_{2,\mu}$ yields
\begin{equation}
1-\varepsilon_{\mathcal M}
-\varepsilon_{\mathcal N}
\leq
\|P_{\mathcal M}Q_{\mathcal N}\|.
\label{eq3}
\end{equation}

On the other hand, the Hilbert--Schmidt norm of
$P_{\mathcal M}Q_{\mathcal N}$ for QDT, defined as \cite{chen2015}
\begin{equation*}
\|P_{\mathcal M}Q_{\mathcal N}\|_{HS}^{2}
=
C_{\mu}^{4}
\int_{\mathcal M}
\int_{\mathcal N}
\left|
E_{\mu}(-\textbf{i} w_1, v_1) \right|^{2}
\left|E_{\mu}(-\textbf{j} w_2, v_2)
 \right|^{2}
\,d\lambda_\mu(v),
\end{equation*}
Using the estimate
$ |E_{\mu}(-\textbf{i} w_1, v_1)|\leq 1,$ and $ |E_{\mu}(-\textbf{j} w_2, v_2)|\leq 1,$
we obtain
\begin{equation*}
\|P_{\mathcal M}Q_{\mathcal N}\|_{HS}^{2}
\leq
C_{\mu}^{4}
|\mathcal M|\,|\mathcal N|.
\end{equation*}

From the classical inequality between the operator norm and the
Hilbert--Schmidt norm, we have

\[
\|P_{\mathcal M}Q_{\mathcal N}\|_{2,\mu}^{2}
\leq
\|P_{\mathcal M}Q_{\mathcal N}\|_{HS}^{2} \leq
C_{\mu}^{4}
|\mathcal M|\,|\mathcal N|.
\]
Consequently,
\begin{equation*}
\|P_{\mathcal M}Q_{\mathcal N}\|_{2,\mu}
\leq
C_{\mu}^{2}
\sqrt{|\mathcal M|\,|\mathcal N|}.
\label{QDT-norm-estimate}
\end{equation*}
Combining this estimate with \eqref{eq3}, we get
\[
1-\varepsilon_{\mathcal M}
-\varepsilon_{\mathcal N}
\leq
C_\mu^{2}
\sqrt{|\mathcal M|\,|\mathcal N|}.
\]

Squaring both sides, we obtain
\[
|\mathcal M|\,|\mathcal N|
\geq
\frac{1}{C_\mu^{\,4}}
\left(
1-\varepsilon_{\mathcal M}
-\varepsilon_{\mathcal N}
\right)^2.
\]
This completes the proof.
\end{proof}

\begin{thm}[Donoho--Stark uncertainty principle for QQPFBT]
Let
$f\in L^2_\mu(\mathbb R,\mathbb H)$ be a quaternion valued signal. Assume that $f$ is
$\varepsilon_{\mathcal M}$-concentrated on a measurable set
$\mathcal M\subseteq \mathbb R$ and
$D^{a,b,c}_{d,e,\mu}[f]$
is $\varepsilon_{\mathcal N}$-concentrated on a measurable set
$\mathcal N\subseteq \mathbb R$.

Then
\begin{eqnarray*}
|\mathcal M||\mathcal N|
\ge
\frac{|b_1 b_2|^{(\mu+1)}}{{c_\mu^4}}
\left(
1-\varepsilon_{\mathcal M}
-\varepsilon_{\mathcal N}
\right)^2.
\end{eqnarray*}
\end{thm}
\begin{proof}
from equatin \eqref{eq3.2}, we have 
\begin{eqnarray*}
\mathbb{D}^{a,b,c}_{d,e,\mu}[f](v)
= \frac{e^{-\textbf{i}\left(c_1w_1^2+e_1w_1\right)}}{(\textbf{i}b_1)^{\mu+1}}
\mathcal{D}^{\mathbb{H}}_{\mu}(g)\!\left(\frac{w_1}{b_1},\frac{w_2}{b_2}\right) \frac{e^{-\textbf{j}\left(c_2 w_2^{2}+e_2 w_2\right)}}{(\textbf{j}b_2)^{\mu+1}},
\end{eqnarray*}
where
\begin{eqnarray*}
g(v_1,v_2)=e^{-\textbf{i}\left(a_1v_1^{2}+d_1v_1\right)}f(v_1,v_2)e^{-\textbf{j}\left(a_2v_2^{2}+d_2v_2\right)}.
\end{eqnarray*}
we obtain
\[
|g(v_1,v_2)|=|f(v_1,v_2)|.
\]
Since $f$ is $\varepsilon_{\mathcal M}$-concentrated on
$\mathcal M \subseteq\mathbb R^2$, then by Definition \eqref{eq4.3}, we have

\[
\left(
\int_{\mathbb R^2\setminus\mathcal M}
|f(v_1,v_2)|^2\,d\lambda_\mu(v)
\right)^{\frac12}
\le
\varepsilon_{\mathcal M}
\|f\|_{2,\mu}.
\]
This implise
\[
\left(
\int_{\mathbb R^{2}\setminus\mathcal M}
|g(v_1,v_2)|^2\,d\lambda_\mu(v)
\right)^{\frac12}
\le
\varepsilon_{\mathcal M}
\|g\|_{2,\mu},
\]
which shows that $g$ is
$\varepsilon_{\mathcal M}$-concentrated on $\mathcal M \subseteq\mathbb R^2$.
Moreover,
\begin{eqnarray*} 
\left|
\mathbb{D}^{a,b,c}_{d,e,\mu}[f](v)
\right|
&=&
\left|
\frac{1}{(b_1)^{\mu+1}}
\mathcal{D}^{\mathbb{H}}_{\mu}(g)\!\left(\frac{w_1}{b_1},\frac{w_2}{b_2}\right) \frac{1}{(b_2)^{\mu+1}}
\right|,\\
\implies 
\left|
\mathbb{D}^{a,b,c}_{d,e,\mu}[f](v)
\right|
&=&
\frac{1}{|b_1 b_2|^{\mu+1}}
\left|\mathcal{D}^{\mathbb{H}}_{\mu}(g)\!\left(\frac{w_1}{b_1},\frac{w_2}{b_2}\right) 
\right|
\end{eqnarray*}

Since $\mathbb{D}^{a,b,c}_{d,e,\mu}[f]$ is $\varepsilon_{\mathcal N}$-concentrated on
$\mathcal N \subseteq\mathbb R^2$,
it follows that
$\mathcal{D}^{\mathbb{H}}_{\mu}(g)\!\left(\frac{w_1}{b_1},\frac{w_2}{b_2}\right)$
is $\varepsilon_{\mathcal N}$-concentrated on $\mathcal N\subseteq\mathbb R_+^2$, i.e $\mathcal{D}_\mu^{\mathbb H}(g)$ is $\varepsilon_{\mathcal N}$-concentrated on $\frac{\mathcal N}{(b_1 b_2)^{\mu+1}}
\subseteq\mathbb R^2.$

Now applying the Donoho-Stark uncertainty principle for the quaternion
Dunkl transform \eqref{leema5.1} to the function $g$, we obtain
\[
|\mathcal M|
\left|
\frac{\mathcal N}{|b_1 b_2|^{\mu+1}}
\right|
\ge
\frac{1}{c_\mu^4}
\left(
1-\varepsilon_{\mathcal M}
-\varepsilon_{\mathcal N}
\right)^2.
\]
which implies
\[
|\mathcal M||\mathcal N|
\ge
\frac{|b_1 b_2|^{\mu+1}}{{c_\mu^4}}
\left(
1-\varepsilon_{\mathcal M}
-\varepsilon_{\mathcal N}
\right)^2.
\]
This completes the proof.
\end{proof}
\section{Signal Recovery in the QQPDT Domain }
In this section, we propose a signal recovery algorithm based on the QQPDT. The proposed algorithm makes use of the localization and concentration properties of the signal in both the original and transform domains to reconstruct the missing information effectively.

Let $g\in L^2_\mu(\mathbb R^{2},\mathbb H)$
be a quaternion-valued signal concentrated on a measurable set
\(\mathcal M\subset \mathbb R^{2}\).
Assume that part of the signal is lost on \(\mathcal M\), and that the observed signal is contaminated by an additive noise
$n\in L^2_\mu(\mathbb R^{2},\mathbb H).$

The received signal is given by
\[
r(v_1,v_2)=
\begin{cases}
g(v_1,v_2)+n(v_1,v_2), & (v_1,v_2)\in \mathcal M^c,\\
0, & (v_1,v_2)\in \mathcal M.
\end{cases}
\]
Assuming that \(n=0\) on \(\mathcal M\), we may write
$$r=(I-P_{\mathcal M})g+n,$$
where \(P_{\mathcal M}\) denotes the time-limiting operator
\[
(P_{\mathcal M}g)(v_1,v_2)
=
\chi_{\mathcal M}(v_1,v_2)g(v_1,v_2).
\]
Let \(Q_{\mathcal N}\) denote the QQPDT frequency-limiting operator
\[
Q_{\mathcal N}(g)
=
\left(
\mathbb{D}^{a,b,c}_{d,e,\mu}
\right)^{-1}
\Bigl(
\chi_{\mathcal N}
\mathbb{D}^{a,b,c}_{d,e,\mu}[g]
\Bigr).
\]
\begin{thm}[Signal recovery in the QQPDT domain]
Let $g\in L^2_\mu(\mathbb R,\mathbb H)$  
be such that $g=Q_{\mathcal N}g.$
Let
$g_n=\sum_{l=0}^{n}
(P_{\mathcal M}Q_{\mathcal N})^l ~r .$
If the measurable sets
\(\mathcal M,\mathcal N\subseteq\mathbb R\)
satisfy the condition
\begin{eqnarray}  
|\mathcal M|^{\frac12}
|\mathcal N|^{\frac12}
<
\frac{|b_1 b_2|^{\mu+1}}
{c_\mu^{2}}, \label{eq6.1}
\end{eqnarray}  
Therefore, the signal \(g\) on
\(x\in\mathcal M\) can be reconstructed using the following algorithm
\begin{eqnarray}   
\left\{\begin{aligned}g_0 &= r,\\
g_{n+1}&=r+P_{\mathcal M}Q_{\mathcal N}g_n .
\end{aligned}
\right.
\end{eqnarray}
Then,
$g_n \rightarrow g$ as   
$ n\rightarrow\infty $
in \(L^2_\mu(\mathbb R,\mathbb H)\). \label{algorithm}
\end{thm}

\begin{proof}
Let
\[
B=(I-P_{\mathcal M}Q_{\mathcal N})^{-1}.
\]
First, we prove that the operator \(B\) is bounded. For this
we establish an upper bound for the Hilbert-Schmidt norm of
$P_{\mathcal M}Q_{\mathcal N}$ for QQPDT,
\begin{equation}
\|P_{\mathcal M}Q_{\mathcal N}\|_{HS}^{2}
=
c_{\mu}^4
\int_{\mathcal M}
\int_{\mathcal N}
\left|
\Psi^{a_1, b_1, c_1}_{d_1, e_1, \mu} (w_1, v_1) \right|^{2}
\left|\Psi^{a_2, b_2, c_2}_{d_2, e_2, \mu}(w_2, v_2)
 \right|^{2}
\,d\lambda_{\mu}(w),
\end{equation}
Using the estimate
$ \left|
\Psi^{a_1, b_1, c_1}_{d_1, e_1, \mu} (w_1, v_1) \right|\leq \frac{1}{|b_1|^{\mu+1}} ,$ and $ \left|\Psi^{a_2, b_2, c_2}_{d_2, e_2, \mu}(w_2, v_2)
 \right|\leq \frac{1}{|b_2|^{\mu+1}} ,$
we obtain
\begin{equation}
\|P_{\mathcal M}Q_{\mathcal N}\|_{HS}^{2}
\leq
\frac{c_{\mu}^4}{|b_1 b_2|^{2\mu+2}}
|\mathcal M|\,|\mathcal N|.
\end{equation}
This implies,
\[
\|P_{\mathcal M}Q_{\mathcal N}\|
\le
\frac{c_\mu^{2}}{|b_1 b_2|^{\mu+1}}
|\mathcal M|^{\frac12}
|\mathcal N|^{\frac12},
\]
and from condition \eqref{eq6.1}, we obtain
\begin{equation*}
\|P_{\mathcal M}Q_{\mathcal N}\| < 1.
\end{equation*}
Hence the operator
$ I-P_{\mathcal M}Q_{\mathcal N}$
is invertible and the series
$\sum_{l=0}^{\infty}
(P_{\mathcal M}Q_{\mathcal N})^l $
converges in operator norm.

Now define
$R_n=\sum_{l=n}^{\infty}
(P_{\mathcal M}Q_{\mathcal N})^l$.
Since \(R_n\) is the remainder of a convergent operator series, then
$\|R_n\|
\longrightarrow 0$
as  $n\rightarrow\infty$.

Moreover,
\[
\begin{aligned}
\lim_{n \to \infty}\|g_n-B~r\|_{\mu,2}
&= \lim_{n \to \infty}
\left\|
\sum_{l=n}^{\infty}
(P_{\mathcal M}Q_{\mathcal N})^l ~ r
\right\|_{\mu,2}
\\
&=\lim_{n \to \infty}
\|R_{n}~r\|_{\mu,2}
\\
&\le
\lim_{n \to \infty}\|R_{n}\|
\,\|r\|_{\mu,2}=0.
\end{aligned}
\]
This implies
\[
\lim_{n\to\infty}
\|g_n-B~r\|_{\mu,2}
=0.
\]

Therefore, $g_n \rightarrow B\,r$
as  $(n\rightarrow\infty)$ in \(L^2_\mu(\mathbb R,\mathbb H)\).
Now we will prove that $B\,r=g$.

Since \(g=Q_{\mathcal N}g\), we have
$r=(I-P_{\mathcal M})g+n
=(I-P_{\mathcal M}Q_{\mathcal N})g+n.$  

Furthermore, by assumption the noise satisfies
\(n=0\) on \(\mathcal M\).
Therefore,
\[
B\,r=(I-P_{\mathcal M}Q_{\mathcal N})^{-1}(I-P_{\mathcal M}Q_{\mathcal N})g=g.\]

Hence
$g_n \rightarrow g$ as   
$ n\rightarrow\infty $
in \(L^2_\mu(\mathbb R,\mathbb H)\).
Thus the information contained in the signal \(g\) on
\(\mathcal M\) can be reconstructed by the iterative algorithm \eqref{algorithm}.
\end{proof}

\section{Conclusion}
In this paper, we proposed QQPDT, a unified quaternion integral transform that combines classical Dunkl transform incorporates with quadratic-phase modulation through ten parameters together with multiplicity parametre $\mu \ge-\frac{1}{2}$, which provides a rich and greter flexible mathematical framework than the the corresponding existing Dunkl transforms. The QQPDT provided a unified framework in which several well known transforms were recovered as special cases through appropriate choices of its parameters. we determined several essential analytical and theoretical results of the proposed transform, including linearity, continuity, scaling, modulation,  Riemann-Lebesgue lemma, Parseval's and inversion formula. Furthermore, a Heisenberg and Donoho-Stark type uncertainty principle for the QQPDT was formulated and proved. In addition, we developed an efficient signal recovery algorithm in the QQPDT setting. This application not only demonstrates the practical usefulness of the QQPDT but also highlights its potential for solving signal reconstruction problems in quaternion valued signal processing. The proposed QQPDT playes an important role for 
substantial progression in harmonic analysis and provides greater flexibility for both theoretical and 
practical applications, particularly in image analysis and signal processing. Moreover, several classes of uncertainty principles for the 
transform remain to be investigated, which opens new directions for future
research. Further extensions may also be explored in higher dimensional algebraic frameworks, such as octonionic, biquaternion and related hypercomplex spaces.


\begin{thebibliography}{99}

\bibitem{mypaper}
M. Y. Bhat and A. H. Dar, \textit{Towards quaternion quadratic-phase Fourier transform}, Mathematical Methods in the Applied Sciences, 48: 10234--10253, (2025).

\bibitem{Safouane2026}
N. Safouane, A. Achak, E. Loualid, I. E. Adha,
\emph{Spectral theorems for the two-sided quaternionic Dunkl transform.}
Complex Analysis and Operator Theory, 20(6): 152, (2026).

\bibitem{Rosler2003}
M. Rösler,
\emph{Dunkl operators: theory and applications.}
In \emph{Orthogonal Polynomials and Special Functions: Leuven 2002}, Springer, Berlin, Heidelberg, 93--135, (2003).

\bibitem{Quaternion book} E. Hitzer, \emph{Quaternion and Clifford Fourier Transforms.} Chapman and Hall/CRC, New York (2021). 

\bibitem{IMAGEPROCESSING}
R. C. Gonzalez, R. E. Woods,
\emph{Digital Image Processing.}
Pearson Education India, (2009).

\bibitem{Dunkl1991}
C.~F.~Dunkl,
\textit{Integral kernels with reflection group invariance},
Can.\ \textbf{j}.\ Math., 43: 1213--1227, (1991). 

\bibitem{deJeu2006}
M. de Jeu,
\emph{Paley--Wiener theorems for the Dunkl transform.}
Transactions of the American Mathematical Society, 358(10): 4225--4250, (2006).

\bibitem{Ghazouani2016}
S. Ghazouani, F. Bouzeffour,
\emph{Heisenberg uncertainty principle for a fractional power of the Dunkl transform on the real line.}
Journal of Computational and Applied Mathematics, 294: 151--176, (2016).

\bibitem{Aprasad}
S.~Varghese, A.~Prasad, and M.~Kundu,
\textit{Properties and applications of quaternion quadratic-phase Fourier transforms},
\textbf{j}.\ Pseudo-Differ.\ Oper.\ Appl., 15: 1-26, (2024).

\bibitem{Tyr2025}
O. Tyr,
\emph{The Beurling theorem for the two-sided quaternionic Dunkl transform.}
Advances in Applied Clifford Algebras, 35(5): 49, (2025).

\bibitem{PrasadKundu2023}
A. Prasad, M. Kundu,
\emph{Spectrum of quaternion signals associated with quaternion linear canonical transform.}
Journal of the Franklin Institute, 361(2): 764--775, (2024).


\bibitem{Bahri2016}
M.~Bahri,
\textit{Uncertainty principles for quaternion-valued transforms},
Appl. Math. Comput., 284: 57--72, (2016).

\bibitem{colourimage}
P. Bas, N. Le Bihan, J. M. Chassery,
\emph{Color image watermarking using quaternion Fourier transform.}
In \emph{2003 IEEE International Conference on Acoustics, Speech, and Signal Processing, 2003. Proceedings (ICASSP'03)}, vol. 3: III-521, (2003).

\bibitem{Kou2016}
K.~Kou,
\textit{Quaternion linear canonical transform},
\textbf{j}. Math. Anal. Appl., 437: 567--588 (2016).

\bibitem{UPQDT}
M. Essenhajy, S. Fahlaoui,
\emph{Uncertainty principles for the two-sided quaternionic Dunkl transform.}
Boletín de la Sociedad Matemática Mexicana, 31(2): 89, (2025).


\bibitem{Saoudi2023}
A.~Saoudi,
\textit{Quadratic-Phase Dunkl Transform: Fundamental properties, translation operators, convolution product and HUP},
arXiv preprint arXiv:2512.22325 (2025).


\bibitem{Raj Kumar} A. Prasad, R. Kumar, 
\emph{Multiresolution analysis and orthonormal linear canonical wavelets in generalized Sobolev spaces.} Journal of Pseudo-Differential Operators and Applications, 4: 1--30, (2025).

\bibitem{TWOSIDEDQDT} {M. Essenhajy, S. Fahlaoui,} {The two-sided quaternionic dunkl transform and hardy’s theorem.}, Rend. Circ. Mat.Palermo,72: 621–633, (2023). 

\bibitem{Shah2021}
F.~A.~Shah, K.~S.~Nisar, W.~Z.~Lone, and A.~Y.~Tantary,
\textit{Uncertainty principles for the quadratic-phase Fourier transforms},
Math.\ Methods Appl.\ Sci., 44: 10416--10431, (2021).

\bibitem{fractionalqdt}
M.~Essenhajy,
\textit{The fractional two-sided quaternionic Dunkl transform and Heisenberg-type inequalities},
arXiv preprint arXiv:2510.11597, (2025), \newblock doi: 10.48550/arXiv.2510.11597.

\bibitem{Bahari2019}
M.~Bahri and R.~Ashino,
\textit{Two-dimensional quaternion linear canonical transform: Properties, convolution, correlation, and uncertainty principle},
Journal of Mathematics, 1--13 (2019).

\bibitem{Bhat2026}
M. Y. Bhat, A. Achak, N. Safouane,
\emph{Titchmarsh theorem for the two-sided quaternionic Dunkl transform.}
Annali dell'Università di Ferrara, 72(1): 1, (2026).

\bibitem{umamaheswari2026}
Umamaheswari S, S. K. Verma and H. Mejjaoli,
\textit{Real Paley–Wiener theorems for the linear canonical Dunkl transform}, Ann. Funct. Anal. 17: 1--25 (2026).

 

\bibitem{qsignal}
M. Bahri, E. S. Hitzer, R. Ashino, and R. Vaillancourt, \textit{Windowed Fourier transform of two-dimensional quaternionic signals}, Applied Mathematics and Computation, 216: 2366--2379, (2010).

\bibitem{Sahbani} J. Sahbani, Quantitative uncertainty principles for the canonical Fourier-Bessel transform, Acta Mathematica Sinica, English Series 38(2), 331--346 (2022).

\bibitem{Dar2023}
A. H. Dar, M. Y. Bhat,
\emph{Donoho-Stark's and Hardy's uncertainty principles for the short-time quaternion offset linear canonical transform.}
Filomat, 37(14): 4467--4480, (2023).

\bibitem{chen2015}
L. P. Chen, K. I. Kou, M. S. Liu,
\emph{Pitt's inequality and the uncertainty principle associated with the quaternion Fourier transform.}
Journal of Mathematical Analysis and Applications, 423(1): 681--700, (2015).




\end{thebibliography}
\end{document}